\documentclass[10pt,reqno]{amsart}%
\usepackage[a4paper,left=2.8cm,right=2.8cm,top=2.8cm,bottom=3.5cm]{geometry}%
\usepackage[T1]{fontenc}%
\usepackage[utf8]{inputenc}%
\usepackage{lmodern,dsfont,graphicx}%
\usepackage[hidelinks,unicode=true]{hyperref}%
\usepackage[shortlabels]{enumitem}%
\usepackage[american]{babel}%
\usepackage{mathtools}
\usepackage{xcolor}
\mathtoolsset{centercolon}
\newcommand\Prob{\mathcal P}
\newcommand\dd{\mathrm d}
\newcommand\Gin{\mathsf{Gin}}
\newcommand\GUE{\mathsf{GUE}}
\newcommand\LUE{\mathsf{LUE}}
\newcommand\mG{{\sigma^{\Gin}_n}}
\newcommand\mH{{\sigma^{\GUE}_n}}
\newcommand\mL{{\sigma^{\LUE}_n}}
\newcommand\lG{\dE\mu_{A^{\Gin}_n}}
\newcommand\lH{\dE\mu_{A^{\GUE}_n}}
\newcommand\lL{\dE\mu_{A^{\LUE}_n}}
\newcommand\oH{{\widetilde H}}
\newcommand\FT[1]{\operatorname{\cF}_{#1}}
\newcommand\CC{\mathbb{C}}
\newcommand\NN{\mathbb{N}}
\newcommand\RR{\mathbb{R}}
\newcommand\ZZ{\mathbb{Z}}
\newcommand\dE{\mathbb{E}}
\newcommand\cE{\mathcal{E}}
\newcommand\cF{\mathcal{F}}
\newcommand\cN{\mathcal{N}}
\newcommand\rB{\mathrm{B}}

\DeclareMathOperator\EE{\dE}
\DeclareMathOperator\En{\cE}
\DeclareMathOperator\ee{e}
\DeclareMathOperator\var{Var}

\DeclareMathOperator\Tr{Tr}

\DeclareMathOperator*\argmin{arg\,min}
\newcommand\ii{\mathrm{i}}
\newcommand\II{\mathds1}
\newcommand\ceq{:=}
\let\originalleft\left
\let\originalright\right
\renewcommand\left{\mathopen{}\mathclose\bgroup\originalleft}
\renewcommand\right{\aftergroup\egroup\originalright}

\numberwithin{equation}{section}%

\title[Monotonicity of the logarithmic energy for random matrices]%
{Monotonicity of the logarithmic energy\\for random matrices}
\author{Djalil Chafa\"{\i}}
\address[DC]{DMA, \'Ecole normale sup\'erieure -- PSL, 45 rue d'Ulm, F-75230 Cedex 5 Paris, France}
\email{\url{mailto:djalil@chafai.net}}
\urladdr{\url{https://djalil.chafai.net/}}

\author{Benjamin Dadoun}%
\address[BD]{Mathematics, Division of Science, New York University Abu Dhabi, UAE}%
\email{\url{mailto:benjamin.dadoun@gmail.com}}%
\urladdr{\url{http://benjamin.dadoun.free.fr/}}%

\author{Pierre Youssef}%
\address[PY]{Mathematics, Division of Science, New York University Abu Dhabi, UAE \& Courant
Institute of Mathematical Sciences, New York University, 251 Mercer st, New York,
NY 10012, USA}%
\email{\url{mailto:yp27@nyu}}%
\urladdr{\url{https://wp.nyu.edu/pyoussef/}}%
\date{Autumn 2022, revised Winter 2024, accepted Spring 2024, to appear in
  RMTA, compiled \today}
\newtheorem{theorem}{Theorem}[section]%
\newtheorem{prob}{Question}[section]%
\newtheorem{lemma}[theorem]{Lemma}%
\theoremstyle{remark}
\begin{document}
\begin{abstract}
    It is well-known that the semi-circle law, which is the limiting distribution in the Wigner theorem, is the minimizer of the logarithmic energy penalized by the second moment. A very similar fact holds for the Girko and Marchenko--Pastur theorems.
    In this work, we shed the light on an intriguing phenomenon suggesting that this functional is monotonic along the mean empirical spectral distribution in terms of the matrix dimension. This is reminiscent of the monotonicity of the Boltzmann entropy along the Boltzmann equation, the monotonicity of the free energy along ergodic Markov processes, and the Shannon monotonicity of entropy or free entropy along the classical or free central limit theorem.
    While we only verify this monotonicity phenomenon for the Gaussian unitary ensemble, the complex Ginibre ensemble, and the square Laguerre unitary ensemble, numerical simulations suggest that it is actually more universal.
    We obtain along the way explicit formulas of the logarithmic energy of the models which can be of independent interest. 
\end{abstract}
\maketitle

\section{Introduction}

The famous Boltzmann H-theorem asserts that the Boltzmann entropy is monotonic in time along the Boltzmann equation which describes the evolution of a probability measure, see for instance~\cite{zbMATH05504088}. The second moment is conserved by the dynamics, and the long-time or at-equilibrium behavior, which is a Gaussian, is the optimizer of the entropy under the conserved constraint. This variational phenomenon is also present for the Kolmogorov evolution equation of ergodic Markov processes through the free energy, which is a Boltzmann entropy penalized by an average energy, see~\cite{zbMATH06542887} and references therein.
Following Shannon, the phenomenon is also present in the classical and free central limit theorems, in which the conservation law is the second moment and the monotonic functional the Boltzmann or the Voiculescu entropy respectively, see for instance~\cite{Artstein04,Madiman07,Shlyakhtenko07} and~\cite{Courtade16,Tao20,Dadoun21}.

Our aim in this work is to initiate the investigation of whether or not similar phenomena arise in classical random matrix limiting theorems such as Wigner semi-circle theorem, Girko circular law theorem, or the Marchenko--Pastur one, see for instance~\cite{zbMATH05645278} for an introduction to random matrix theory. 
For an~$n\times n$ matrix~$A$ with entries in~$\CC$, we denote by
$\lambda_1(A),\ldots,\lambda_n(A)$ its eigenvalues, which are the roots in~%
$\CC$ of its characteristic polynomial
$z\mapsto\det(A-z\mathrm{Id}_n)$, viewed as a multiset of~$n$ elements. We
denote the empirical spectral distribution by 
\[
\mu_A\ceq\frac{1}{n}\sum_{k=1}^n\delta_{\lambda_k(A)}.
\]
Let us first start by looking at the Wigner semi-circular theorem which states that given a Hermitian random matrix $M_n=(\xi_{ij})_{1\leq i,j\leq n}$ in which the upper-triangular entries are independent, identically distributed centered complex random variables with unit variance, while the diagonal entries are i.i.d.\ real random variables with bounded mean and variance, the sequence of (random) empirical spectral distributions  
$
(\mu_{\frac{1}{\sqrt{n}}M_n})_{n\in \NN}
$
 converges almost surely to the Wigner semi-circular distribution given by
\[
\mu_{\mathrm{sc}}:=\frac{1}{2\pi}\sqrt{{(4-x^2)}_+}\,\dd x. 
\]
In analogy with what was discussed previously, we may ask if this convergence is driven by similar phenomena in the sense that there is some characterizing functional which is monotone in terms of the matrix dimension~$n$. 
The choice of such functional could be guided by two considerations: a variational characterization and a large deviation principle.
This naturally leads us in this context to look at the logarithmic energy, defined below,
which is closely related to Voiculescu's free entropy. Indeed, Ben Arous and Guionnet~\cite{BenArous97} showed 
in particular that Gaussian Wigner matrices, and more generally matrices drawn from the $\beta$-ensemble,
satisfy a large deviation principle with rate function related to the logarithmic energy. Moreover, it is known that Wigner semi-circular law can be characterized through a variational principle involving the logarithmic energy.

Let $\mathcal{P}_2(K)$ be the set of probability measures on~$K$ where~%
$K=\RR$ or~$K=\CC$, with finite second moment
$m_2\ceq\int|x|^2\dd\mu(x)<\infty$. The logarithmic energy
$\En\colon\Prob_2(K)\to(-\infty,\infty]$ is defined by
\begin{equation*}
  \En(\mu)
  \ceq\iint\log\frac{1}{|x-y|}\,\mu(\dd x)\mu(\dd y),
\end{equation*}
which makes sense since
$(x,y)\in\CC^2\mapsto\frac{|x|^2+|y|^2}{2}-\log|x-y|$ is bounded below
by a finite constant. When~$K=\RR$,
then $-{\En}$ is the Voiculescu free entropy of free probability theory~\cite{Voiculescu05,Hiai00}. 

For~$\lambda>0$, the minimization of~$\En$ over
$\mathcal{P}_2(K)$ under the constraint $m_2(\mu)=\lambda$ is equivalent to
the minimization over $\mathcal{P}_2(K)$ of the penalized logarithmic energy
\begin{equation*}
  \mu\mapsto\En_{\lambda}^{(2)}(\mu)
  \ceq\frac{1}{2\lambda}m_2(\mu)+\En(\mu).
\end{equation*}
Indeed, by using
$\En(\mu)=-\frac{1}{2}\log\tfrac{\alpha}{\lambda}+\En(\mathrm{dil}_{\sqrt{\frac{\lambda}{\alpha}}}\mu)$
and
$m_2(\mathrm{dil}_{\sqrt{\frac{\alpha}{\lambda}}}\mu)=\frac{\alpha}{\lambda}m_2(\mu)$,
we get
\begin{align*}
 \min_{\mu}\En_{\lambda}^{(2)}(\mu)&=\min_{\alpha>0}\left(\frac{\alpha}{2\lambda}+\min_{m_2(\mu)=\alpha}\En(\mu)\right)\\
	&=\min_{\alpha>0}\left(\frac{\alpha}{2\lambda}-\frac{1}{2}\log\frac{\alpha}{\lambda}\right)%
 +\min_{m_2(\mu)=\lambda}\En(\mu)\\
 &=\frac{1}{2}+\min_{m_2(\mu)=\lambda}\En(\mu).
\end{align*}
The functional $\En_{\lambda}^{(2)}$ is strictly convex, lower semi-continuous
with compact sub-level-sets\footnote{With respect to the weak convergence of
  probability measures for continuous and bounded test functions.}, see~\cite{BenArous97,Saff97,Hiai00}. 
With these definitions in hand, the semi-circular distribution is characterized by 
\[
\mu_{\mathrm{sc}} =\argmin_{\substack{\mu\in\mathcal{P}_2(\RR)\\m_2(\mu)=1}}\En
  =\argmin_{\mu\in\mathcal{P}_2(\RR)}\En_1^{(2)}.
\]

This leads us to investigate whether $\En_1^{(2)}$ is monotone along Wigner semi-circular theorem. Since the corresponding measures are random in this case, a more suitable formulation would be to look at the expected empirical measures which corresponds to the weak form of Wigner's theorem, i.e., when the convergence of the sequence of empirical measures is considered in expectation. Then the goal of this paper is in particular to investigate the following question:

\begin{prob}\label{qu: semi-circle}
Let~$M_n$ be a random Wigner matrix, $\mu_n\ceq\mu_{\frac{1}{\sqrt{n}} M_n}$ be the corresponding empirical spectral distribution and suppose that $m_2(\dE\mu_n)=1$. Is it true that $\En_1^{(2)}(\dE\mu_n)\searrow \En_1^{(2)}(\mu_{\mathrm{sc}})$?
\end{prob}

As stated, the above problem would be impossible to study in full generality as considerations on the finiteness of $\En_1^{(2)}(\dE\mu_n)$ for every~$n$ would enter into play. A potential modification would be to additionally impose that $\En_1^{(2)}(\dE\mu_n)$ is uniformly bounded. However, this would rule out several models of random matrices such as Rademacher random matrices for which simulations suggest a positive answer to Question~\ref{qu: semi-circle} provided one conditions on an appropriately chosen event (see Figure~\ref{fig:wigner}). 
The fascinating aspect in the above problem is that the monotonicity is considered at every fixed finite dimension and not just in the large dimension regime. This is related to the increment $\mathbb{E}(\mu_{n+1}-\mu_n)$ as well as to  the ``gradient'' of the convex functional $\En_1^{(2)}$, which is a logarithmic potential.
 
A classical random matrix model for which the above technical considerations are not needed is the Gaussian Unitary Ensemble. An~$n\times n$ random Hermitian matrix $A^{\GUE}_n$ belongs to
the normalized Gaussian unitary ensemble, if its density is proportional to 
$A\mapsto\ee^{-\frac{n}{2}\Tr(A^2)}$. We have that 
\begin{equation*}
  m_2(\lH)=1
  \quad\text{and}\quad
  \lH
  \xrightarrow[n\to\infty]{\mathrm{weak}}
  \frac{{\sqrt{(4-x^2)}_+}}{2\pi}\,\dd x.
\end{equation*}

Our first result states that Question~\ref{qu: semi-circle} has a positive solution in this case. 

\begin{theorem}[Monotonicity and convexity for Gaussian unitary
  ensemble]\label{th:GUE}
  For every~$n\geq1$, 
  \[
    \En_1^{(2)}(\lH)
    =\frac{3}{4}+\frac{1}{2}\left(\log n+\gamma+\frac{1}{2n}-H_n\right),
  \]
  where $\gamma\ceq0.577\ldots$ denotes the Euler constant and $H_n\ceq\sum_{k=1}^n\frac{1}{k}$ is the $n$-th harmonic number. 
  In particular
  $n\in\{1,2,\ldots\}\mapsto\En_{1}^{(2)}(\lH)$
  is strictly decreasing and strictly convex.
\end{theorem}

Note that we refer to a real sequence ${(u_n)}_{n\geq1}$ as (strictly) convex when $n\mapsto u_{n+1}-u_n$ is (strictly) increasing. While our motivation in this paper consists in capturing the monotonicity phenomenon, we believe that having a closed form of the logarithmic energy of a GUE matrix is interesting in its own right. We are unaware of the existence of any such formulas in the literature. These explicit formulas are a new and intriguing manifestation of the determinantal integrability of GUE.

We now turn to another classical random matrix limiting theorem which is the circular law one. As before, our starting point is the following variational characterization of the circular law 
\[ 
  \mu_{\mathrm{circ}}\ceq\frac{\mathds{1}_{|z|\leq1}}{\pi}\,\dd z
  =\argmin_{\substack{\mu\in\mathcal{P}_2(\CC)\\m_2(\mu)=\frac{1}{2}}}\En
  =\argmin_{\mu\in\mathcal{P}_2(\CC)}\En_{\frac{1}{2}}^{(2)}.
\]

 In analogy to our previous discussion, we are led to investigate whether $\En_{\frac12}^{(2)}$ is monotone along the circular law theorem, and thus we formulate the following question: 
 
 \begin{prob}\label{qu: circle}
 Let $M_n\ceq(\xi_{ij})_{1\leq i,j\leq n}$ be an~$n\times n$ random matrix whose entries are i.i.d.\ complex random variables of unit variance and denote by $\mu_n\ceq\mu_{\frac{1}{\sqrt{n}} M_n}$ the corresponding empirical spectral distribution. Is it true that $\En_{\frac12}^{(2)}(\dE\mu_n)\searrow \En_{\frac12}^{(2)}(\mu_{\mathrm{circ}})$?
 \end{prob}

Let $A^{\Gin}_n$ be a random~$n\times n$ matrix drawn from the complex Ginibre ensemble, with density proportional to $A\mapsto\ee^{-n\Tr(AA^*)}$ where $A^*\ceq\overline{A}^\top$ is the conjugate-transpose. Note that (see~\eqref{eq: m2-gin})
\begin{equation*}
  m_2(\lG)=\frac{1}{2}+\frac{1}{2n}
  \quad\text{and}\quad
  \lG
  \xrightarrow[n\to\infty]{\mathrm{weak}}
\mu_{\mathrm{circ}}.
\end{equation*}

The next theorem states that Question~\ref{qu: circle} has a positive solution in this case.

\begin{theorem}[Monotonicity and convexity for complex Ginibre ensemble]\label{th:Ginibre}
  For every~$n\geq1$, 
  \begin{align*}
    \En_{\frac{1}{2}}^{(2)}(\lG)
    &=\frac{3}{4}
    +\frac{1}{2}\left(\log n+\gamma+\frac{1}{2n}-H_n\right)
      +\frac{1}{2}\sum_{k=n+1}^\infty\frac{4^{-k}\binom{2k}k}{k(k-1)}\\
    &=\En_{1}^{(2)}(\lH)
      +\frac{1}{2}\sum_{k=n+1}^\infty\frac{4^{-k}\binom{2k}k}{k(k-1)},    
  \end{align*}
  in particular
  $n\in\{1,2,\ldots\}\mapsto\En_{\frac{1}{2}}^{(2)}(\lG)$
  is strictly decreasing and strictly convex.
\end{theorem}
The fact that the (penalized) logarithmic energy of the complex Ginibre ensemble can be expressed in terms of that of Gaussian Unitary Ensemble is intriguing. It would be interesting to understand the meaning of the additional term that appears in the above expression.

Finally, we consider the square Marchenko--Pastur law given by $\mu_{\mathrm{MP}}\ceq \frac{{\sqrt{(4-x)}_+}}{2\pi \sqrt{x}}\,\dd x$. This law appears in particular as the limiting spectral distribution of $\frac{1}{n}X_nX_n^*$ where $X_n=(\xi_{ij})_{1\leq i,j\leq n}$ is an~$n\times n$ random matrix whose entries are i.i.d.\ centered complex random variables of unit variance. The starting point is again a variational characterization for $\mu_{\mathrm{MP}}$ (see~\cite[Proposition~5.5.13]{Hiai00}):
\[
\mu_{\mathrm{MP}}= \argmin_{\substack{\mu\in\mathcal{P}_2(\RR_+)\\m_1(\mu)=1}}\En
  =\argmin_{\mu\in\mathcal{P}_2(\RR_+)}\En^{(1)},
\]
where 
\[
\En^{(1)}(\mu) \ceq \En(\mu)+m_1(\mu).
\]
In view of this, we state the following question:

 \begin{prob}\label{qu: MP}
Let $X_n\ceq(\xi_{ij})_{1\leq i,j\leq n}$ be an~$n\times n$ random matrix whose entries are i.i.d.\ centered complex random variables of unit variance, and let $\mu_n\ceq\mu_{\frac{1}{n}X_nX_n^*}$ be the corresponding spectral distribution. Is it true that $\En^{(1)}(\dE\mu_n)\searrow \En^{(1)}(\mu_{\mathrm{MP}})$?
 \end{prob}
 
 Note that we only stated the square version of the problem.  A straightforward generalization of the above would be to consider rectangular matrices for which the ratio of the two dimensions is constant. It is conceivable that one can study the rectangular case in more generality where two dimensional parameters will enter into play. It is however unclear at this stage how the monotonicity property would be measured in terms of both dimensions.

 In this paper, we will show that Question~\ref{qu: MP} has a positive solution in the case of (square) Laguerre Unitary Ensemble. More precisely, set $A_n^{\LUE}\ceq(A_n^{\Gin}){(A_n^{\Gin})}^*$ and note that 
 \begin{equation*}
  m_1(\lL)=1
  \quad\text{and}\quad
  \lL
  \xrightarrow[n\to\infty]{\mathrm{weak}}
  \mu_{\mathrm{MP}}.
\end{equation*}

 \begin{theorem}[Monotonicity and convexity for square Laguerre Unitary ensemble]\label{th:laguerre}
For every~$n\geq 1$, 
\begin{align*}
    \En^{(1)}(\lL)
    &=\frac32+\left(\log n+\gamma+\frac1{2n}-H_n\right)\\
    &=2\En_{1}^{(2)}(\lH),    
  \end{align*}
  in particular
  $n\in\{1,2,\ldots\}\mapsto\En^{(1)}(\lL)$
  is strictly decreasing and strictly convex.
  \end{theorem}

The above theorem also implies that $\En(\lL)=2\En(\lH)$. It is rather fascinating that this holds in all dimensions. While the proof carries heavy calculations, it would be very interesting if there is an easy argument proving that the logarithmic energy of LUE is twice that of GUE. 

Figure~\ref{fig:gue-gin-lue} illustrates Theorems~\ref{th:GUE},~\ref{th:Ginibre}, and~\ref{th:laguerre} respectively. 

\begin{figure}[htbp]
  \includegraphics[width=\textwidth]{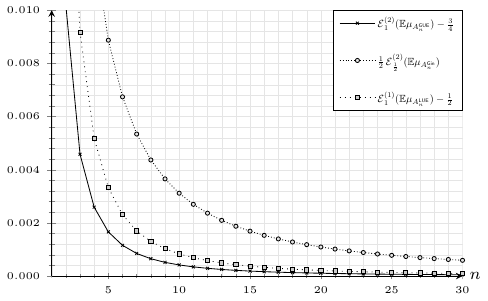}
  \caption{Logarithmic energy of GUE, Ginibre, and LUE.\label{fig:gue-gin-lue}}
\end{figure}

Following~\cite{zbMATH05730452,zbMATH05645278}, a $\beta$-ensemble is the probability measure on~$S^n$ of the form
\begin{equation}\label{eq:betaens}
 \frac{\mathrm{e}^{-n\beta\sum_{i=1}^nV(x_i)}}{Z_{V,\beta,n}}
 \prod_{i<j}|x_i-x_j|^\beta\,
 \dd x_1\cdots\dd x_n
\end{equation}
where~$S$ is a linear subset of $\mathbb{C}$, $\mathrm{d}x$ is the Lebesgue measure on~$S$, $\beta>0$ is a parameter, $V:S\to\mathbb{R}$ is a function, and $Z_{V,\beta,n}$ is the normalizing constant. Viewed as a Boltzmann--Gibbs measure, it is a Coulomb or log-gas. Under a mild assumption on~$V$, a suitable version of the Laplace method implies that as~$n\to\infty$, the mean empirical measure $\mathbb{E}[\frac{1}{n}\sum_{i=1}^n\delta_{x_i}]$ tends to the equilibrium measure 
\begin{equation*}
\mu_V:=\argmin_\mu\left(\int V\mathrm{d}\mu+\mathcal{E}(\mu)\right).
\end{equation*}
When~$\beta=2$, the density in~\eqref{eq:betaens} has a determinantal integrable structure associated with a kernel, which produces in particular an explicit formula for the density of the mean empirical distribution.  
\begin{center}
\medskip
\begin{tabular}{c|c|c|c}
 Model & $S$ & $\beta$ & $V(t)$\\ \hline
 CUE & $\{z\in\mathbb{C}:|z|=1\}$ & $2$ & $\equiv1$\\
 GUE & $\mathbb{R}$ & $2$ & $t^2$\\
 LUE (square) & $\mathbb{R}_+$\!\!\!\! & $2$ & \!\!\!\!\:$t$\\
 Ginibre (complex) & $\mathbb{C}$ & $2$ & ${|t|}^2$
\end{tabular}
\medskip
\end{center}
This is why we focus on GUE, LUE, and the complex Ginibre ensemble. Note that the case of the Circular Unitary Ensemble (CUE) is trivial from our perspective since its mean empirical spectral distribution is already equal to the uniform distribution and does not depend on the dimension.

Our computations for GUE and LUE remind \cite[Sec.~2 and 6]{zbMATH02060143},
but we were not aware of that.

We emphasize that what we compute is not what is considered for instance in
\cite{MR3046995} and in \cite{byun2023progress}.

Yet another determinantal random matrix model is the Elliptic Ginibre
Ensemble, which interpolates between Ginibre and GUE, see for instance
\cite{AIHPA_1998__68_4_449_0}, \cite[eq.~(1.2)--(1.5)]{MR2594353}, and
\cite{byun2023progress,francois-garcia-zelada}. It is natural to explore the
monotonicity phenomenon in this case. We leave that to the interested readers.

The universality of this monotonicity phenomenon, beyond these exactly
solvable random matrix models, remains for now mysterious. Nevertheless, we
believe that disregarding some technical considerations, Questions~\ref{qu:
  semi-circle},~\ref{qu: circle} and~\ref{qu: MP} have positive answers. We
provide suggestive numerical experiments for general Wigner and Girko random
matrices with non Gaussian i.i.d.\ entries, as well as some Dumitriu--Edelman
beta ensembles of random matrices.

\bigskip

The paper is organized as follows: in Section~\ref{sec:GUE}, we prove Theorem~\ref{th:GUE}, while Theorems~\ref{th:Ginibre} and~\ref{th:laguerre} are proven in Sections~\ref{sec:Ginibre} and~\ref{sec:LUE}. Finally, simulations are included in Section~\ref{se:numexp}. 

\medskip

{\small\textbf{Acknowledgments.} The authors are grateful to Peter Forrester
  and Michel Ledoux for their feedback. Part of this work was conducted during
  a visit of BD and PY to Département de Mathématiques et Applications of
  École Normale Supérieure -- PSL. The numerical experiments were carried out
  on the High Performance Computing resources at New York University Abu
  Dhabi.}

\section{Preliminaries}

\subsection{General notations}
As one may guess from the theorem statements, our computations will involve the classical harmonic numbers
\begin{equation}
  H_n\ceq1+\frac12+\cdots+\frac1n.
\label{eq:harmonic_numbers}
\end{equation}
These numbers appear in the log-derivative of the (generalized) binomial coefficients
\begin{equation}
  \binom zn\ceq\frac{z(z-1)\cdots(z-n+1)}{n!}=\frac{\Gamma(z+1)}{\Gamma(n+1)\Gamma(z+1-n)},
  \label{eq:binomial}
\end{equation}
where~$\Gamma$ denotes the gamma function. By the lowercase~$\gamma$ we mean Euler--Mascheroni's constant
\begin{align}
  \gamma&\ceq\lim_{n\to\infty}\Bigl(H_n-\log n\Bigr)=-\Gamma'(1),
  \label{eq:mascheroni}
\intertext{of which Dirichlet obtained the helpful integral representation~\cite[p.~248, Example~2]{Whittaker96}}
\gamma&=\int_0^\infty\frac{\dd t}t\left(\frac1{1+t}-\ee^{-t}\right).
  \label{eq:reprgamma}
\intertext{%
Combining this with the integral representation of the logarithm}
  \log a&=\int_0^\infty\frac{\dd t}t\Bigl(\ee^{-t}-\ee^{-ta}\Bigr),\qquad\Re a>0
  \label{eq:reprlog}
\end{align}
(see, e.g.,~\cite[\S~6.222, Example~6]{Whittaker96})
we derive the following alternative expression for the log-energy.
\begin{lemma}[Exponential weight formulation of the logarithmic energy]\label{lem:logenergy}
  Let $\nu\in\Prob_2(K)$, where $K\in\{\RR,\CC\}$,
  and recall $\En(\nu)\ceq
    -{\iint\log{\lvert x-y\rvert}\,\nu(\dd x)\nu(\dd y)}$.
  Then, for every~$p>0$,
  \[
    \En(\nu)
    =\frac\gamma p-\frac1p\int_0^\infty\frac{\dd t}t\left(\frac 1{1+t}-\iint\ee^{-t{\lvert x-y\rvert}^p}\nu(\dd x)\nu(\dd y)\right).
  \]
\end{lemma}
\begin{proof}
  First, we have
  \begin{align*}
    \En(\nu)
      &=-\frac1p\iint\log{{\lvert x-y\rvert}^p}\,\nu(\dd x)\nu(\dd y)\\[.4em]
      &=-\frac1p\iint\II_{\{\lvert x-y\rvert>1\}}\log{{\lvert x-y\rvert}^p}\,\nu(\dd x)\nu(\dd y)
      +\frac1p\iint\II_{\{\lvert x-y\rvert\le1\}}(-{\log{{\lvert x-y\rvert}^p}})\,\nu(\dd x)\nu(\dd y),
    \intertext{where both integrands are nonnegative and only the second may give rise to an infinite integral
    (because $\log{\lvert x-y\rvert}\le\log\sqrt{(1+|x|^2)(1+|y|^2)}\le(|x|^2+|y|^2)/2$ and $\nu\in\Prob_2(K)$).
    Plugging in~\eqref{eq:reprlog} and exchanging the order of integrations with Fubini-Tonelli's theorem then leads to}
    \En(\nu)&=-\frac1p\int_0^\infty\frac{\dd t}t
    \iint\II_{\{\lvert x-y\rvert>1\}}
      \left(\ee^{-t}-\ee^{-t{\lvert x-y\rvert}^p}\right)\nu(\dd x)\nu(\dd y)\\[.4em]
      &\hphantom{={}}-\frac1p\int_0^\infty\frac{\dd t}t
      \iint\II_{\{\lvert x-y\rvert\le1\}}
        \left(\ee^{-t}-\ee^{-t{\lvert x-y\rvert}^p}\right)\nu(\dd x)\nu(\dd y)\\[.4em]
    &=-\frac1p\int_0^\infty\frac{\dd t}t\left(\ee^{-t}-\iint\ee^{-t{\lvert x-y\rvert}^p}\nu(\dd x)\nu(\dd y)\right)
  \end{align*}
  (keeping in mind that the integral may evaluate to~$-\infty$).
We inject~\eqref{eq:reprgamma} to conclude.
\end{proof}

\subsection{Level densities of determinantal processes}

Let us recall the explicit expressions for~$\lH$, $\lG$, and~$\lL$, derived from~\eqref{eq:betaens}, see~%
$\mH$~\cite[\S~6.2]{Mehta04},
$\mG$~\cite[\S~15.1]{Mehta04}, and~%
$\mL$~\cite[\S~5.7, with the weight $w\ceq\II_{\RR_+}\exp^{-1}$]{Mehta04}.
Specifically (taking the normalizations into account):
\begin{enumerate}[label=(\alph*),wide]
  \item The mean empirical spectral distribution of the normalized Gaussian Unitary Ensemble is
\begin{flalign}
&&&&\lH(\dd x)=\sqrt\frac n2\,\mH\left(x\sqrt\frac n2\right)\,\dd x,
&&\text{for}\quad
\mH(x)\ceq\frac{\ee^{-x^2}}{n\sqrt\pi}
	  \sum_{k=0}^{n-1}\frac{{h_k(x)}^2}{2^k\,k!},&&x\in\RR,&
\label{eq:levelGUE}
\end{flalign}
where
$h_k(x)\ceq{(-1)}^k\,{\ee^{x^2}}\frac{\dd^k}{\dd x^k}\ee^{-x^2}$, $k\in\ZZ_+$,
denote the (physicist's) Hermite polynomials, which are
the orthogonal polynomials with respect to the weight $\ee^{-x^2}\dd x$ on~$\RR$:
\begin{equation}
  \int_\RR h_i(x)h_j(x)\ee^{-x^2}\dd x
    =2^i\,i!\sqrt\pi\,\II_{\{i=j\}},\quad i,j\in\ZZ_+;\label{eq:orthhermite}
\end{equation}
\item The mean empirical spectral distribution of the normalized Ginibre ensemble is
\begin{flalign}
  &&&&\lG(\dd z)
  =\sqrt n\,\mG\left(z\sqrt n\right)\,\dd z,
  &&\text{for}\quad
  \mG(z)\ceq\frac{\ee^{-\lvert z\rvert^2}}{n\pi}
      \sum_{k=0}^{n-1}\frac{{\lvert z\rvert}^{2k}}{k!},&&z\in\CC;&
  \label{eq:levelGin}
\end{flalign}
\item The mean empirical spectral distribution of the normalized Laguerre Unitary Ensemble is
\begin{flalign}
  &&&&\lL(\dd s)=n\,\mL(ns)\,\dd s,
  &&\text{for}\quad
  \mL(s)\ceq\frac{\ee^{-s}}n\sum_{k=0}^{n-1}{L_k(s)}^2,&&s\in\RR_+,&
\label{eq:levelLUE}
\end{flalign}
where $L_k(s)\ceq\frac{\ee^s}{k!}\frac{\dd^k}{\dd s^k}(\ee^{-s}s^k)$, $k\in\ZZ_+$,
denote the classical Laguerre polynomials, which are orthogonal
with respect to~$\ee^{-s}\dd s$ on~$\RR_+$:
\begin{equation}
    \int_0^\infty L_i(s)L_j(s)\ee^{-s}\dd s=\II_{\{i=j\}},\qquad i,j\in\ZZ_+.%
    \label{eq:laguerre_orthogonality}
\end{equation}
\end{enumerate}

\section{Logarithmic energy of GUE}\label{sec:GUE}
The goal of this section is to prove Theorem~\ref{th:GUE}.
The first step is to reformulate the level density~$\mH$
in~\eqref{eq:levelGUE} in terms of a linear combination
of the Hermite polynomials~$h_{2k}$:
\begin{lemma}\label{lem:GUE density}
For all~$n\ge1$ and~$x\in\RR$,
\[
\mH(x)=\frac{\ee^{-x^2}}{n\sqrt\pi}\sum_{k=0}^{n-1}\frac{\binom n{k+1}}{2^k\,k!}\,h_{2k}(x).
\]
\end{lemma}
\begin{proof}
We plug in Feldheim's identity on Hermite polynomials~\cite[Eq.\ (1.5)]{Feldheim38},
\[{h_k(x)}^2=2^k\,k!\sum_{r=0}^k\frac{\binom kr}{2^r\,r!}\,h_{2r}(x),\]
into the expression~\eqref{eq:levelGUE} to get
\[\mH(x)=\frac{\ee^{-x^2}}{n\sqrt\pi}\sum_{k=0}^{n-1}\sum_{r=0}^k\frac{\binom kr}{2^r\,r!}\,h_{2r}(x).\]
The lemma is now immediate if we recall the hockey-stick identity
\[\sum_{k=0}^{n-1}\binom kr=\binom{n}{r+1}\]
and switch the summations.
\end{proof}
We now express the logarithmic energy
$\En(\lH)$
in terms of the level density~$\mH$:
\begin{lemma}\label{lem:logenergy-gue}
  We have
  \[
    \En(\lH)
    =\frac12\log{\frac n2}+\frac\gamma2-\frac12
    \int_0^\infty\frac{\dd t}t\left(\frac t{1+t}
    -\iint\ee^{-\frac{{\lvert x-y\rvert}^2}t}
    \mH(x)\mH(y)\,\dd x\dd y\right),
  \]
  where~$\gamma$ is Euler--Mascheroni's constant~\eqref{eq:mascheroni}.
\end{lemma}
\begin{proof}
  By~\eqref{eq:levelGUE} and the change of variable $x\gets x\sqrt{2/n}$, we have
  \[
    \En(\lH)
      =\frac12\log{\frac n2}+\En(\nu),
  \]
  where $\nu(\dd x)\ceq\mH(x)\,\dd x$. It remains to apply Lemma~\ref{lem:logenergy} with~$p=2$
  and perform the change of variable $t\gets 1/t$.
\end{proof}
Next, we calculate the double integral appearing in Lemma~\ref{lem:logenergy-gue}.
\begin{lemma}\label{lem:integral-mh}
For all~$n\ge1$ and~$t>0$,
\[
  \iint\ee^{-\frac{{\lvert x-y\rvert}^2}t}\mH(x)\mH(y)\,\dd x\dd y
  =\frac{\sqrt{2t}}{n^2}\sum_{0\le k,\ell<n}\!\!\!\!{(-1)}^{k+\ell}
  \frac{\binom n{k+1}}{k!}\frac{\binom n{\ell+1}}{\ell!}
  \frac{(2k+2\ell)!}{(k+\ell)!}{(2t+4)}^{-k-\ell-\frac12}.
\]
\end{lemma}
\begin{proof}
Expanding~$\mH$ with Lemma~\ref{lem:GUE density},
\begin{align}
  &\hspace{-2em}\iint\ee^{-\frac{{\lvert x-y\rvert}^2}t}\mH(x)\mH(y)\,\dd x\dd y
  \notag\\[.4em]
  &=\frac1{\pi n^2}
  \sum_{0\le k,\ell<n}\!\!
    \frac{\binom N{k+1}}{2^k\,k!}\frac{\binom N{\ell+1}}{2^\ell\,\ell!}
    \int\ee^{-x^2}h_{2k}(x)\left(\int\ee^{-y^2}h_{2\ell}(y)
    \ee^{-\frac{\lvert x-y\rvert^2}t}\dd y\right)\dd x.\label{eq:integral-mh1}
\end{align}
Fixing now~$k$ and~$\ell$, we identify the innermost integral
$C(x)\ceq\int\ee^{-y^2}h_{2\ell}(y)
\ee^{-\frac{\lvert x-y\rvert^2}t}\dd y$ as a convolution
product. Its Fourier transform equals%
\footnote{Our convention for the Fourier transform is
$\displaystyle\widehat f(\xi)\ceq\FT{y}[f(y)](\xi)\ceq
  \frac1{\sqrt{2\pi}}\int_{-\infty}^\infty f(y)\ee^{-\ii y\xi}\,\dd y,\quad
\xi\in\RR$.}
\begin{align*}
\widehat C(\xi)
&=\sqrt{2\pi}\,
\FT{y}{\left[{\ee^{-y^2}}H_{2\ell}(y)\right]}(\xi)
  \cdot\FT{y}{\left[\ee^{-\frac{y^2}t}\right]}(\xi)\\
&={{(-1)}^\ell\sqrt\pi\,\xi^{2\ell}\ee^{-\frac{\xi^2}4}}
  \cdot{\sqrt{\frac t2}\ee^{-\frac{t\xi^2}4}}\\
&={(-1)}^\ell\sqrt{\frac{\pi t}2}\,\xi^{2\ell}\ee^{-\frac{1+t}4\xi^2}
\intertext{by the convolution theorem (for the expression of $\FT{y}[{\ee^{-y^2}}h_{2\ell}(y)](\xi)$, see, e.g.,~\cite[Theorem~4.6.6]{Ismail09}),
so}
C(x)
=\widehat{\widehat C}(-x)
&={(-1)}^\ell\sqrt{\frac{\pi t}2}\,\FT{\xi}\left[\xi^{2\ell}\ee^{-\frac{1+t}4\xi^2}\right](-x)\\
&={(1+t)}^{-\ell-\frac12}\sqrt{\pi t}\ee^{-\frac{x^2}{1+t}}
  H_{2\ell}\left(\frac x{\sqrt{1+t}}\right)
\end{align*}
by the inversion theorem (and the fact that~$h_{2\ell}$
is an even polynomial). It follows that
\begin{align}
&\hspace{-2em}\int\ee^{-x^2}h_{2k}(x)\left(\int\ee^{-y^2}h_{2\ell}(y)
    \ee^{-\frac{\lvert x-y\rvert^2}t}\dd y\right)\dd x\notag\\[.4em]
&={(1+t)}^{-\ell-\frac12}\,\sqrt{\pi t}
  \int{\ee^{-\frac{2+t}{1+t}x^2}}h_{2k}(x)h_{2\ell}\left(\frac x{\sqrt{1+t}}\right)\dd x%
\notag\\[.4em]
&={(1+t)}^{-\ell}\,\sqrt{\frac{\pi t}{2+t}}
\int{\ee^{-x^2}}h_{2k}\left(x\sqrt{\frac{1+t}{2+t}}\right)
  h_{2\ell}\left(\frac x{\sqrt{2+t}}\right)\dd x\notag\\[.4em]
&=\pi\frac{{(1+t)}^{-\ell}\sqrt t}{{(2+t)}^{k+\ell+\frac12}}
\sum_{\substack{0\le i\le k\\0\le j\le\ell}}
{(-1)}^{i+j}\binom{2k}{2i}\binom{2\ell}{2j}\frac{(2i)!}{i!}\frac{(2j)!}{j!}\,{(1+t)}^{k-i+j}\,
4^{k-i}(2k-2i)!\II_{k-i=\ell-j}\notag\\[.4em]
&=\pi{(-1)}^{k+\ell}\,(2k)!(2\ell)!\,{(2+t)}^{-k-\ell-\frac12}\sqrt t\sum_{n=0}^{k\wedge\ell}
\frac{4^n}{(2n)!(k-n)!(\ell-n)!}\notag\\[.4em]
&=\pi{(-1)}^{k+\ell}\,\frac{(2k+2\ell)!}{(k+\ell)!}\,{(2+t)}^{-k-\ell-\frac12}\sqrt t,\label{eq:integral-mh2}
\end{align}
where the last equality is due to Gauss's theorem~\cite[Theorem~1.3 for $_2F_1(-k,-\ell;\frac12;1)$]{Bailey64}, and the third equality
has been obtained by applying, together with the orthogonality
relations~\eqref{eq:orthhermite}, the scaling formula~\cite[Eq.\ (4.6.33)]{Ismail09}:
\[
h_{2p}(xy)=\sum_{m=0}^p{(-1)}^m\,y^{2p-2m}{(1-y^2)}^m\binom{2p}{2m}\frac{(2m)!}{m!}\,h_{2p-2m}(x),
\quad p\in\ZZ_+.%
\]
The proof is completed by plugging~\eqref{eq:integral-mh2} back into~\eqref{eq:integral-mh1}.
\end{proof}

Before we are ready to prove Theorem~\ref{th:GUE}, we need one last calculation.
\begin{lemma}\label{lem:calcul-GUE}
For every integer~$n\geq 1$, we have 
\[
\frac1{2n^2}\sum_{\substack{0\le k,\ell<n\\k+\ell\neq0}}\frac{{(-1)}^{k+\ell}}{k+\ell}\binom n{k+1}\binom n{\ell+1}{\binom{k+\ell}k}
= \frac14+\frac12 \left(\frac{1}{2n}- H_n\right).
\]
\end{lemma}
\begin{proof}
Denoting by~$S_n$ the corresponding sum stated above,
we split it according to whether
$k=0$ or~$k\ge1$:
\[
S_n
=n\sum_{\ell=1}^{n-1}\frac{{(-1)}^\ell}\ell\binom n{\ell+1}
+n\sum_{k=1}^{n-1}\frac{{(-1)}^k}k\binom n{k+1}
  \sum_{\ell=0}^{n-1}{(-1)}^\ell\binom{n-1}\ell\frac{\binom{k+\ell-1}{k-1}}{\ell+1},
\]
where we have used that
$\binom n{\ell+1}\binom{k+\ell}k/(k+\ell)
=n\binom{n-1}\ell\binom{k+\ell-1}{k-1}/k(\ell+1)$. But for~$k\neq 1$, because
$\binom{k+\ell-1}{k-1}/(\ell+1)=P_k(\ell)$ with
$p_k(x)\ceq(x+2)\cdots(x+k-1)/(k-1)!$ a polynomial of degree $k-2<n-1$, the
rightmost sum above is~$0$ by the theory of finite differences~\cite[Trick~5.3.2]{Graham94}. Then
\begin{align*}
  S_n&=n\sum_{\ell=1}^{n-1}\frac{{(-1)}^\ell}\ell\binom n{\ell+1}-\frac{n^2(n-1)}2
       \sum_{\ell=0}^{n-1}\frac{{(-1)}^\ell}{\ell+1}\binom{n-1}\ell\\
     &=\left(\frac12-H_n\right)n^2+\frac n2,
\end{align*}
from elementary combinatorial identities on harmonic numbers~\cite[\S~6.4]{Graham94}, such as
\begin{equation*}
  \sum_{\ell=1}^n\frac{{(-1)}^{\ell-1}}\ell\binom n\ell=H_n.
\end{equation*}
To conclude, it then suffices to divide~$S_n$ by~$2n^2$.
\end{proof}

We now have everything in hand to prove Theorem~\ref{th:GUE}. 
\begin{proof}[Proof of Theorem~\ref{th:GUE}]
Combining Lemmas~\ref{lem:logenergy-gue} and~\ref{lem:integral-mh}, we can write 
\begin{align*}
\hspace{-1em}\En(\lH)&=\frac12\log{\frac n2}+\frac\gamma2\\[.4em]
&\qquad-\int_0^\infty\frac{\dd t}{2t}\left(\frac t{1+t}-\frac{\sqrt{2t}}{n^2}\sum_{k,\ell=0}^{n-1}{(-1)}^{k+\ell}
  \frac{\binom n{k+1}}{k!}\frac{\binom n{\ell+1}}{\ell!}
  \frac{(2k+2\ell)!}{(k+\ell)!}{(2t+4)}^{-k-\ell-\frac12}\right)
  \dd t.
\end{align*}
Now, observe that in the displayed sum, the term corresponding to $k=\ell=0$
reduces to $\sqrt{t/(2+t)}$; this term yields a divergent integral and will
precisely compensate with the non-integrable part $t/(1+t)$. The other
summands for $k+\ell=: s\ge1$ are however integrable and will involve the Beta
integrals
\[\int_0^\infty{(2t+4)}^{-s-\frac12}\,{(2t)}^{-\frac12}\,\dd t
\underset{\scriptscriptstyle t=\frac{2u}{1-u}}=2^{-1-2s}
\int_0^1u^{-\frac12}{(1-u)}^{s-1}\,\dd u
\enspace=\enspace\frac1{2s\binom{2s}s}.\]
Specifically, we arrive at
\begin{align*}
 \En(\lH)
  &=\frac12\log{\frac{n}{2}}+\frac\gamma2
  -\int_0^\infty\frac{\dd t}{2t}
  \left(\frac t{1+t}-\sqrt{\frac t{2+t}}\right)\\[.4em]
  &\hspace{4em}+\frac1{2n^2}\sum_{\substack{0\le k,\ell<n\\k+\ell\neq0}}\frac{{(-1)}^{k+\ell}}{k+\ell}\binom n{k+1}\binom n{\ell+1}{\binom{k+\ell}k}.
\end{align*}
Finally, simplifying with
\[\int_0^\infty\frac{\dd t}{2t}
\left(\frac t{1+t}-\sqrt{\frac t{2+t}}\right)
={\left[\frac12\log{\frac{1+t}{2+t}}
-\log\left(1+\sqrt\frac t{2+t}\right)\right]}_{t=0}^\infty
=\frac12\log{\frac12}\]
and Lemma~\ref{lem:calcul-GUE},
we get
\begin{equation*}
\En(\lH)
  =\frac14
  +\frac12\left(\log n
  +\gamma
  +\frac1{2n}
  -H_n\right).
\end{equation*}
It remains to recall $m_2(\lH)=1$ to
obtain the formula promised in Theorem~\ref{th:GUE}.
Finally, since
$n\mapsto 2\bigl[\En(\mathbb{E}\mu_{A^{\GUE}_{n+1}})-\En(\lH)\bigr]=\log(n+1)-\log n-\frac{1+2n}{2n(n+1)}$ is
negative and strictly increasing, we get that $n\mapsto\En(\lH)$ is strictly
decreasing and strictly convex.
\end{proof}

\section{Logarithmic energy of complex Ginibre}\label{sec:Ginibre}
The goal of this section is to prove Theorem~\ref{th:Ginibre}.
Recalling~\eqref{eq:levelGin}, we have
\begin{flalign*}
  &&&&&&\lG(\dd z)
  =\sqrt n\,\mG\left(z\sqrt n\right)\,\dd z,
  &&\text{with}\quad
  \mG(z)=\frac{\Gamma(n,|z|)}{n\pi},&&&&z\in\CC;&
\end{flalign*}
where we introduced the \emph{upper incomplete} Gamma function
(see~\cite[Eq.\ (15.1.16)]{Mehta04})
\begin{equation}\label{eq:uigam}
  \Gamma(n,r)\ceq\int_r^\infty
  x^{n-1}\ee^{-x}\dd x
  =(n-1)!\ee^{-r}\sum_{k=0}^{n-1}\frac{r^k}{k!}.
\end{equation}
From this, we have that for any integrable function $f\colon\CC\to\RR$,
\begin{align}
  \int_\CC f(z)\,\lG(\dd z)
  &=\frac1{2\pi}
    \int_0^{2\pi}\dd\theta\int_0^\infty
    f\left(\sqrt\frac rn\ee^{\ii\theta}\right)\frac{\Gamma(n,r)}{n!}\,\dd r.\label{eq:intf}
\end{align}
Using this for instance with $f=\left|\cdot\right|^2$ gives
\begin{equation}\label{eq: m2-gin}
  \int_\CC|z|^2\,\lG(\dd z)
  =\frac1{n^2}\sum_{k=0}^{n-1}\frac1{k!}\int_0^\infty r^{k+1}\ee^{-r}\dd r
  =  \frac1{n^2}\sum_{k=0}^{n-1}\frac{(k+1)!}{k!}=\frac1{n^2}\sum_{k=1}^nk=\frac12+\frac1{2n}.
\end{equation}

We start by reducing the calculation of the logarithmic energy to that of the expectation of some random variable. 
\begin{lemma}\label{lem: start-gin}
For every integer~$n\geq 1$, denote by $B_n\ceq\min\bigl(\frac{X_n}{Y_n},\frac{Y_n}{X_n}\bigr)$
where $(X_n,Y_n)$ is a pair of i.i.d.\ Gamma$(n+1,1)$ variables, i.e.,
the density of~$X_n$ is given by $\frac{x^{n}}{n!}\ee^{-x}\II_{[0,\infty)}$. Then
\[
\En(\lG)=\frac{\log n}2
  +\frac{1+\gamma-H_n}{2}
  +\frac14\EE{\log\left(B_n\ee^{-B_n}\right)},
\]
where we recall that~$H_n$ is the $n$-th harmonic number
and~$\gamma$ is Euler--Mascheroni's constant,
see~\eqref{eq:harmonic_numbers} and~\eqref{eq:mascheroni}.
\end{lemma}
\begin{proof}
Note that 
\begin{equation*}
  \int_0^{2\pi}\int_0^{2\pi}\log{\left\lvert\sqrt r\ee^{\ii\theta}-\sqrt s\ee^{\ii\varphi}\right\rvert}\,\dd\theta\dd\varphi
  =2\pi^2\log{\max(r,s)},\qquad r,s>0,
\end{equation*}
which is an easy consequence of the standard formula
$\int_0^{2\pi}\log{\lvert{r-\ee^{\ii\theta}\rvert}}\,\dd\theta=2\pi\log\max(r,1)$.
Using this and~\eqref{eq:intf}, we can write 
\begin{align*}
 \En(\lG)
 &=-{\iint_{\CC^2}\log{\lvert x-y\rvert}\,\lG(\dd x)\,\lG(\dd y)}\\[.4em]
 &=\frac{\log n}2
  -\frac12\iint_{{(0,\infty)}^2}\!\!\!\frac{\Gamma(n,r)\Gamma(n,s)}{{n!}^2}\log{\max(r,s)}\,\dd r\dd s.
\end{align*}
The integral formula for the upper incomplete Gamma function in~\eqref{eq:uigam}
and Fubini's theorem give
\begin{equation*}
\En(\lG)
=\frac{\log n}{2}-\frac12\iint_{{[0,\infty)}^2}\frac{{(xy)}^{n-1}
    \ee^{-x-y}}{{n!}^2}
  \left(\int_0^x\dd r\int_0^y\dd s\log{\max(r,s)}\right)\dd x\dd y.
\end{equation*}
We observe at this point that $xy\,\frac{{(xy)}^{n-1}\ee^{-x-y}}{{n!}^2}$ is the
joint density of a pair $(X_n,Y_n)$ of independent and identically distributed
Gamma$(n+1,1)$ random variables.
In order to use this fact, we force an extra factor~$xy$ as follows:
\begin{align*}
\int_0^x\dd r\int_0^y\dd s\log{\max(r,s)}
  &=\tfrac12\,{\min(x,y)}^2+xy\Bigl(\log\max(x,y)-1\Bigr)\\
  &=\tfrac12\,xy\min\Bigl(\tfrac xy,\tfrac yx\Bigr)
    +xy\Bigl(\tfrac12\log{\max\Bigl(\tfrac xy,\tfrac yx\Bigr)}
    +\tfrac{1}2\log{xy}-1\Bigr)\\
  &=xy\Bigl(-1+\frac{\log(x)+\log(y)}2\Bigr)
    -\tfrac12 xy\log\Bigl(\min(\tfrac xy,\tfrac yx)
    \ee^{-\min(\frac xy,\frac yx)}\Bigr),
\end{align*}
where we have intentionally used the quantities~$x/y$ and~$y/x$, in such a way
that if we consider the random variables $Z_n\ceq X_n/Y_n$ and
$B_n\ceq\min(Z_n,1/Z_n)$, then
\begin{equation*}
  \En( \lG)
  =\frac{\log n}{2}
  +\frac12
  -\frac12\EE{\frac{\log(X_n)+\log(Y_n)}2}
  +\frac14\EE{\log\left(B_n\ee^{-B_n}\right)}.
\end{equation*}
The log-moment of the Gamma$(n+1,1)$ distribution is known
to be
\begin{equation*}
  \EE{\log{X_n}}=\EE{\log{Y_n}}=H_n-\gamma
\end{equation*}
(it suffices to consider the derivative of $\EE{X^t}=\frac{\Gamma(n+1+t)}{n!}$ at~$t=0$,
which amounts to taking the log-derivative of $\Gamma(n+1+t)=(n+t)\cdots(1+t)\Gamma(1+t)$ at~$t=0$,
and to recall that $\gamma=-\Gamma'(1)$).
The lemma follows. 
\end{proof}

Next, we calculate the corresponding expectation appearing in the previous lemma. 
\begin{lemma}\label{lem: calcul-gin}
Let $B_n\ceq\min\big(\frac{X_n}{Y_n}, \frac{Y_n}{X_n}\big)$ where $(X_n,Y_n)$ is a pair of i.i.d.\ Gamma$(n+1,1)$ random variables.
Then 
\begin{equation*}
  \EE{\log\left(B_n\ee^{-B_n}\right)}
  =-\frac{n+1}n+2\sum_{k=n+1}^\infty\frac{4^{-k}\binom{2k}k}{k(k-1)}.
\end{equation*}
\end{lemma}
\begin{proof}
Note that a ratio of two independent Gamma$(n+1,1)$ random variables follows the ``Beta prime distribution''~\cite[p.~248]{Johnson95} with parameters $(n+1,n+1)$; its probability density
function is
\begin{equation*}
  x\in[0,\infty)\mapsto\frac1{\rB(n+1,n+1)}\,x^n{(1+x)}^{-2n-2},
\end{equation*}
where
\begin{equation}\label{eq:betan+1}
  \rB(n+1,n+1)\ceq\int_0^1x^n{(1-x)}^n\,\dd x=\frac1{(n+1)\binom{2n+1}n}.
\end{equation}
Therefore, denoting $Z_n=\frac{X_n}{Y_n}$, we have 
\begin{align}
  \EE{\log\left(B_n\ee^{-B_n}\right)}
  &=\EE{\II_{\{Z_n<1\}}\log\left(Z_n\ee^{-Z_n}\right)}
  +\EE{\II_{\{Z_n^{-1}<1\}}\log\left(Z_n^{-1}\ee^{-Z_n^{-1}}\right)}\notag\\
  &=2\EE{\II_{\{Z_n<1\}}\log\left(Z_n\ee^{-Z_n}\right)}\notag\\
  &=2(n+1)\binom{2n+1}n
  \int_0^1\log\left(x\ee^{-x}\right)\,x^n{(1+x)}^{-2n-2}\,\dd x.
  \label{eq:logterm}
\end{align}
The change of variable $u=x/(1+x)$ which maps~$[0,1]$ to~$[0,\frac12]$
and $\log(x\ee^{-x})=\log x-x$ give
\begin{equation}\label{eq:logtermint}
  \int_0^1\log\left(x\ee^{-x}\right)\,x^n{(1+x)}^{-2n-2}\,\dd x
  =\int_0^{\frac12}\log\left(\frac u{1-u}\right)u^n{(1-u)}^n\,\dd u
  -\int_0^{\frac12}u^{n+1}{(1-u)}^{n-1}\,\dd u.
\end{equation}
The last integral is an incomplete Beta integral which can be computed by
integration by parts as
\begin{align}
\int_0^{\frac12}u^{n+1}{(1-u)}^{n-1}\,\dd u
  &=\frac{n+1}n\int_0^{\frac12}u^n{(1-u)}^n\,\dd u-{\left[u^{n+1}\frac{{(1-u)}^n}n\right]}_{0}^{\frac12}\notag\\
  &=\frac{n+1}{2n}\int_0^1u^n{(1-u)}^n\,\dd u-\frac1{n2^{2n+1}}\notag\\[.4em]
  &=\frac1{2n\binom{2n+1}n}-\frac1{n2^{2n+1}}, \label{eq:logtermint1}
\end{align}
where we used crucially the symmetry of $u^n{(1-u)}^n$ around~$u=\frac12$ and~\eqref{eq:betan+1}. The first integral on the right-hand side of~\eqref{eq:logtermint} equals
\begin{align*}
	\int_0^{\frac12}\log\left(\frac u{1-u}\right)u^n{(1-u)}^n\,\dd u
	&=\int_0^{\frac14}v^n{(1-4v)}^{-\frac12}\log\left(\frac{1-\sqrt{1-4v}}{1+\sqrt{1-4v}}\right)\,\dd v\notag\\
    &=\sum_{k=0}^\infty\binom{2k}k\int_0^{\frac14}v^{k+n}
      \log\left(\frac{1-\sqrt{1-4v}}{1+\sqrt{1-4v}}\right)\,\dd v\notag\\
    &=-{\sum_{k=0}^\infty\frac{\binom{2k}k}{k+n+1}\int_0^{\frac14}v^{k+n}
      {(1-4v)}^{-\frac12}\,\dd v}\notag\\
    &=-{\sum_{k,\ell=0}^\infty\frac{4^{-(k+\ell+n+1)}\binom{2k}k\binom{2\ell}\ell}{(k+n+1)(k+\ell+n+1)}},
\end{align*}
where the first equality is due to the change of variable $v=u(1-u)$, the
second and fourth equalities come from the Newton series expansion
${(1-4v)}^{-\frac12}=\sum_{k=0}^\infty\binom{2k}kv^k$ (valid for~%
$|v|<\frac14$) and Fubini's theorem, and the third equality follows from
integration by parts (the logarithmic term has derivative $1/v\sqrt{1-4v}$).
Gathering the terms of this last series by~$k+\ell$ leads to
\begin{align}
  \int_0^{\frac12}\log\left(\frac u{1-u}\right)u^n{(1-u)}^n\,\dd u
  &=-\sum_{k=0}^\infty
    \frac{4^{-(k+n+1)}}{k+n+1}\sum_{r=0}^k\frac{\binom{2r}r\binom{2k-2r}{k-r}}{r+n+1}\notag\\
  &=-\frac12\sum_{k=0}^\infty\frac{4^{-(k+n+1)}\binom{2k+2n+2}{k+n+1}}{(k+n+1)(n+1)\binom{2n+1}n}\notag\\
  &=-\frac1{2(n+1)\binom{2n+1}n}\sum_{k=n+1}^\infty\frac{\binom{2k}k}{4^k\,k},\label{eq:logtermint2}
\end{align}
where the second equality follows from
the combinatorial identities~\cite[(3.95) \& (Z.45)]{Gould72}
\begin{equation*}
  \sum_{r=0}^k\binom{2r}{r}\binom{2k-2r}{k-r}\frac{n+1}{n+1+r}=2^{2k}\frac{\binom{n+k+\frac12}k}{\binom{n+k+1}k}
  =\frac{\binom{2n+2k+2}{n+k+1}}{\binom{2n+2}{n+1}}.
\end{equation*}
Plugging $\eqref{eq:logtermint}=\eqref{eq:logtermint2}-\eqref{eq:logtermint1}$
back into~\eqref{eq:logterm}, we get
\begin{equation*}
  \mathbb{E}[\log(B_n\ee^{-B_n})]
  =\frac{(n+1)\binom{2n+1}n}{4^n\,n}
  -\frac{n+1}n
  -\sum_{k=n+1}^\infty\frac{\binom{2k}k}{4^k\,k}.
\end{equation*}
We can slightly simplify this last expression if we notice that
\begin{equation*}
  \frac{\binom{2k}k}k
  =\frac{\binom{2k+1}k-\binom{2k}{k-1}}k
  =\frac{\binom{2k+2}{k+1}}{2k}-\frac{\binom{2k}k}{k+1},
\end{equation*}
so
\begin{align*}
  \sum_{k=n+1}^\infty\frac{\binom{2k}k}{4^k\,k}
  &=\sum_{k=n+1}^\infty\frac{\binom{2k+2}{k+1}}{2^{2k+1}\,k}
    -\sum_{k=n+1}^\infty\frac{4^{-k}\binom{2k}k}{k+1}\\
  &=2\sum_{k=n+2}^\infty\frac{4^{-k}\binom{2k}k}{k-1}
    -\frac{\binom{2n+1}{n+1}}{4^n},
    \intertext{by reindexation for the first sum and using for the second sum the Gauss
    summation theorem~\cite[Theorem~1.3 for the hypergeometric series
	${}_2F_1(1,n+\frac32;n+3;1)$]{Bailey64}. Reintroducing the
	term at $k=n+1$ and subtracting twice the left-hand side on both sides, we arrive at}
    -{\sum_{k=n+1}^\infty\frac{\binom{2k}k}{4^k\,k}}
  &=2\sum_{k=n+1}^\infty\frac{4^{-k}\binom{2k}k}{k(k-1)}
    -\frac{(n+1)\binom{2n+1}n}{4^n\,n},
\end{align*}
which implies the desired identity. 
\end{proof}

We are now ready to prove Theorem~\ref{th:Ginibre}.
\begin{proof}[Proof of Theorem~\ref{th:Ginibre}]
Combining~\eqref{eq: m2-gin} together with Lemmas~\ref{lem: start-gin} and~\ref{lem: calcul-gin}, we have 
\[
 \En_{\frac{1}{2}}^{(2)}(\lG)= \underbrace{\frac34
  +\frac12\left(\log n+\gamma+\frac1{2n}-H_n\right)\vphantom{\sum_{k=n+1}^\infty}}_{a_n}
+\underbrace{\frac12\sum_{k=n+1}^\infty\frac{4^{-k}\binom{2k}k}{k(k-1)}}_{b_n},
\]
which proves the first part of the theorem. 
Finally, we have already seen at the very end of the proof of Theorem~\ref{th:GUE} that $n\mapsto a_n$ is decreasing and strictly convex. Moreover,
$n\mapsto b_n$ is obviously strictly decreasing, and it is easily checked that
\begin{equation*}
  (b_{n+1}-b_n)-(b_n-b_{n-1})
  =\frac{(2n)!}{(n+1)!^2(n^2-1)}\left[4{(n+1)}^3-(2n+2)(2n+1)(n-1)\right]
  >0.
\end{equation*}
This concludes the proof. 
\end{proof}

\section{Logarithmic energy of LUE}\label{sec:LUE}
The goal of this section is to prove Theorem~\ref{th:laguerre}.
The first step is to reformulate the level density~$\mL$
in~\eqref{eq:levelLUE} in terms of a linear combination
of the Laguerre polynomials.
Let us record here their explicit expression
(easily derived from Leibniz' differentiation formula):
\begin{equation}
    L_k(s)=\sum_{i=0}^k\frac{{(-1)}^i}{i!}\binom kis^i,\label{eq:laguerre_closedform}
\end{equation}
see, e.g.,~\cite[Theorem~4.6.1]{Ismail09}.
\begin{lemma}\label{lem:LUE density}
For every~$s\in\RR_+$,
\[
\mL(s)=2\ee^{-s}\sum_{k=0}^{n-1}{(-1)}^{n-k}\binom{n-1}k\binom{k-\frac12}n L_{2k}(2s).
\]
\end{lemma}
\begin{proof}
We plug in Howell's identity on Laguerre polynomials~\cite[(7.1)]{Howell37},
\[
{L_k(s)}^2=\frac1{4^k}\sum_{r=0}^k\binom{2r}r\binom{2k-2r}{k-r}L_{2r}(2s),
\]
into the expression~\eqref{eq:levelLUE} to get
\begin{align*}
    \mL(s)&=\frac1n\ee^{-s}
    \sum_{r=0}^{n-1}\frac{\binom{2r}r}{4^r}\left(\sum_{k=0}^{n-1-r}\frac{\binom{2k}k}{4^k}\right)
      L_{2r}(2s)\\[.4em]
    &=\frac1n\ee^{-s}
    \sum_{r=0}^{n-1}\frac{\binom{2r}r}{4^r}\binom{n-r-\frac12}{n-r-1}
        L_{2r}(2s),
\end{align*}
where the appearance of the generalized binomial coefficient (see~\eqref{eq:binomial})
results from~\cite[(1.109)]{Gould72}.
The proof is completed using the alternative expression
\[4^{-r}\binom{2r}r\binom{n-r-\frac12}{n-r-1}=2n{(-1)}^{n-r}\binom{n-1}r\binom{r-\frac12}n\]
given by~\cite[(Z.46) \& (Z.55)]{Gould72}.
\end{proof}

We next derive the following expansion of the logarithmic energy of LUE. 
\begin{lemma}\label{lem: step1-lue}
We have
\[
\En(\lL)
=\log(2n)+\gamma
    -\int_0^\infty\frac{\dd t}t
    \left(\frac1{1+t}-J(t)
    \right),
\]
where 
\begin{align*}
J(t)\ceq2
\sum_{0\le k,\ell<n}
{(-1)}^{k+\ell}&\binom{n-1}k\binom{n-1}\ell\binom{k-\frac12}n\binom{\ell-\frac12}n\\
&\times 
    \int_0^\infty\ee^{-(t+\frac12)x}L_{2k}(x)\,\dd x
    \int_0^x\ee^{-(\frac12-t)y}L_{2\ell}(y)\,\dd y.
\end{align*}
\end{lemma}
\begin{proof}
  By~\eqref{eq:levelLUE} and the change of variable $x\gets x/2n$, we have
  \[
      \En(\lL)
        =\log(2n)+\En(\nu),
    \]
    where $\nu(\dd x)\ceq\frac12\mL(\frac x2)\,\dd x$. Applying Lemma~\ref{lem:logenergy} with~$p=1$,
  it follows that
  \[
   \En(\lL)
    =\log(2n)+\gamma
    -\int_0^\infty\frac{\dd t}t
    \left(\frac1{1+t}-\frac14\iint_{{(0,\infty)}^2}\ee^{-t{\lvert x-y\rvert}}
    \mL\left(\frac x2\right)\mL\left(\frac y2\right)\,\dd x\dd y
    \right).
  \]
    The previous double integral over~${(0,\infty)}^2$
    is twice the same integral over $\{0<y<x\}$ by symmetry of~$x$ and~$y$,
    and plugging in the expression in Lemma~\ref{lem:LUE density}, we get
\begin{align*}
\frac14\iint_{{(0,\infty)}^2}\ee^{-t{\lvert x-y\rvert}}
    \rho_n\left(\frac x2\right)\rho_n\left(\frac y2\right)\dd x\dd y
=2
\sum_{0\le k,\ell<n}&
{(-1)}^{k+\ell}\binom{n-1}k\binom{n-1}\ell\binom{k-\frac12}n\binom{\ell-\frac12}n\label{eq:logterm_part}\\
&\times
    \int_0^\infty\ee^{-(t+\frac12)x}L_{2k}(x)\,\dd x
    \int_0^x\ee^{-(\frac12-t)y}L_{2\ell}(y)\,\dd y.
\end{align*}
This finishes the proof. 
\end{proof}

Next, we calculate the double integral appearing in~$J(t)$ above. 
We first need the following binomial identity. 
\begin{lemma}\label{lem:binomial_identity}
    For all integers $k,\ell,p\ge0$,
    \[\sum_{i=0}^\ell{(-1)}^{i+k}\binom{\ell}{i+p}\binom ik
    =\II_{\{k=\ell,p=0\}}+\II_{\{\ell>k,p\ge1\}}\binom{\ell-k-1}{p-1}.\]
\end{lemma}
\begin{proof}
    Let~$k\ge0$, and let $x_i\ceq{(-1)}^i\binom ik$
    and $y_i\ceq{(-1)}^k\,\II_{\{i=k\}}$ for~$i\ge0$.
    We have
    \begin{align*}
    x_\ell&=\sum_{i=0}^\ell{(-1)}^{\ell-i}\binom\ell iy_i
    \intertext{for all~$\ell\ge0$, so Pascal's inversion formula yields}
    y_\ell&=\sum_{i=0}^\ell\binom\ell ix_i=
    \sum_{i=0}^\ell{(-1)}^i\binom\ell i\binom ik.
    \end{align*}
    This proves the case~$p=0$.
    We now assume~$p\ge1$ and proceed by induction on~$\ell\ge0$. The identity
    is clear for~$\ell=0$. Let us assume~$\ell\ge1$. Then Pascal's triangle gives
    \begin{align*}
        &\sum_{i=0}^\ell{(-1)}^{i+k}\binom{\ell}{i+p}\binom ik\\
        &\qquad=\sum_{i=0}^{\ell-1}{(-1)}^{i+k}\binom{\ell-1}{i+p}\binom ik
        -\sum_{i=0}^{\ell-1}{(-1)}^{i+k}\binom{\ell-1}{i+p-1}\binom ik\\[.4em]
        &\qquad=\II_{\{\ell>k+1\}}\binom{\ell-k-2}{p-1}+
        \II_{\{\ell=k+1,p=1\}}+\II_{\{\ell>k+1,p\ge2\}}\binom{\ell-k-2}{p-2}\\[.4em]
        &\qquad=\II_{\{\ell>k\}}\binom{\ell-k-1}{p-1},
    \end{align*}
    in virtue of the base cases and the induction hypothesis. This establishes the claim.
\end{proof}

\begin{lemma}\label{lem: step2-lue}
For all $k,\ell\in\ZZ_+$ and~$t>0$, we have
\[
 \int_0^\infty\ee^{-(t+\frac12)x}L_{2k}(x)\,\dd x
    \int_0^x\ee^{-(\frac12-t)y}L_{2\ell}(y)\,\dd y= \frac{\II_{\{k=\ell\}}}{t+\frac12}-\II_{\{k>\ell\}}\frac{{\left(t-\frac12\right)}^{2k-2\ell-1}}{{\left(t+\frac12\right)}^{2k-2\ell+1}}.
\]
\end{lemma}
\begin{proof}
Suppose~$t\neq\frac12$. Repeated integration by parts allows us to write
\[\int_0^x\ee^{-(\frac12-t)y}\,L_{2\ell}(y)\,\dd y
=-\sum_{p=0}^{2\ell}{\left(\frac12-t\right)}^{\!-p-1}
{\left[\ee^{-(\frac12-t)y}L^{(p)}_{2\ell}(y)\right]}_{y=0}^{y=x},\]
where~\eqref{eq:laguerre_closedform} gives the $p$-th derivative of~$L_{2\ell}$:
\[
L^{(p)}_{2\ell}(y)={(-1)}^p\sum_{i=0}^{2\ell}\frac{{(-1)}^i}{i!}\binom{2\ell}{i+p}y^i.
\]
Therefore
\begin{align*}
\int_0^x\ee^{-(\frac12-t)y}L_{2\ell}(y)\,\dd y
&=\sum_{p=0}^{2\ell}{\left(t-\frac12\right)}^{\!-p-1}
\left[
\ee^{-(\frac12-t)x}\sum_{i=0}^{2\ell}\frac{{(-1)}^i}{i!}\binom{2\ell}{i+p}
x^i-\binom{2\ell}p\right]\\[.4em]
&=\sum_{p=0}^{2\ell}\frac1{{\left(t-\frac12\right)}^{p+1}}
\sum_{i=0}^{2\ell}\frac{{(-1)}^i}{i!}\binom{2\ell}{i+p}x^i\ee^{-(\frac12-t)x}
-\frac{{\left(t+\frac12\right)}^{2\ell}}{{\left(t-\frac12\right)}^{2\ell+1}},
\end{align*}
where the last equality is due to the binomial theorem. 
Since
\begin{align*}
\int_0^\infty\ee^{-(t+\frac 12)x}L_{2k}(x)\,\dd x
&=\frac{{\left(t-\frac12\right)}^{\!2k}}{{\left(t+\frac12\right)}^{\!2k+1}}
\end{align*}
(see, e.g.,~\cite[4.11.(25), p.~174]{Erdelyi54}),
we get that for all $k,\ell\in\ZZ_+$ and~$t>0$,
\begin{align*}
    &\int_0^\infty\ee^{-(t+\frac12)x}L_{2k}(x)\,\dd x
    \int_0^x\ee^{-(\frac12-t)y}\,L_{2\ell}(y)\,\dd y\\
    &\qquad=\sum_{p=0}^{2\ell}\frac1{{\left(t-\frac12\right)}^{p+1}}
    \sum_{i=0}^{2\ell}\frac{{(-1)}^i}{i!}\binom{2\ell}{i+p}
    \int_0^\infty\ee^{-x}x^{i}\,L_{2k}(x)\,\dd x
    -\frac{{\left(t-\frac12\right)}^{2k-2\ell-1}}{{\left(t+\frac12\right)}^{2k-2\ell+1}}.
\end{align*}
Next, we combine the orthogonality relations~\eqref{eq:laguerre_orthogonality}
with the identity
\[\frac{x^i}{i!}=\sum_{j=0}^i{(-1)}^j\binom ijL_j(x)\]
(which can be derived from~\eqref{eq:laguerre_closedform} by Pascal's inversion formula)
to get 
\[\int_0^\infty\ee^{-x}x^i\,L_{2k}(x)\,\dd x=
i!\sum_{j=0}^i{(-1)}^j\binom ij\II_{\{j=2k\}}
=i!\binom i{2k}.\]
Thus
\begin{align*}
&\int_0^\infty\ee^{-(t+\frac12)x}L_{2k}(x)\,\dd x
\int_0^x\ee^{-(\frac12-t)y}L_{2\ell}(y)\,\dd y\\[.4em]
&\qquad=\sum_{p=0}^{2\ell}\frac1{{\left(t-\frac12\right)}^{p+1}}
\sum_{i=0}^{2\ell}{(-1)}^i\binom{2\ell}{i+p}\binom i{2k}
-\frac{{\left(t-\frac12\right)}^{2k-2\ell-1}}{{\left(t+\frac12\right)}^{2k-2\ell+1}}.
\end{align*}
Note that $\binom i{2k}\neq0$ implies~$\ell\ge k$,
since we are summing over $0\le i\le2\ell$.
In fact, we have
\[\sum_{i=0}^{2\ell}{(-1)}^i\binom{2\ell}{i+p}\binom i{2k}
=\II_{\{k=\ell,p=0\}}+\II_{\{\ell>k,p\ge1\}}\binom{2\ell-2k-1}{p-1}\]
by Lemma~\ref{lem:binomial_identity} above, so
\begin{align*}
\sum_{p=0}^{2\ell}\frac1{{\left(t-\frac12\right)}^{p+1}}
\sum_{i=0}^{2\ell}{(-1)}^i\binom{2\ell}{i+p}\binom i{2k}
&=\frac{\II_{\{k=\ell\}}}{t-\frac12}+\frac{\II_{\{\ell>k\}}}{{(t-\frac12)}^2}{\left(1+\frac1{t-\frac12}\right)}^{\!2\ell-2k-1}\\[.4em]
&=\frac{\II_{\{k=\ell\}}}{t-\frac12}+\II_{\{\ell>k\}}\frac{{\left(t-\frac12\right)}^{2k-2\ell-1}}{{\left(t+\frac12\right)}^{2k-2\ell+1}},
\end{align*}
where we used the binomial theorem again. 
Hence
\begin{align*}
    &\int_0^\infty\ee^{-(t+\frac12)x}L_{2k}(x)\,\dd x
    \int_0^x\ee^{-(\frac12-t)y}L_{2\ell}(y)\,\dd y\\[.4em]
    &\qquad=\frac{\II_{\{k=\ell\}}}{t-\frac12}+\II_{\{\ell>k\}}\frac{{\left(t-\frac12\right)}^{2k-2\ell-1}}{{\left(t+\frac12\right)}^{2k-2\ell+1}}
    -\frac{{\left(t-\frac12\right)}^{2k-2\ell-1}}{{\left(t+\frac12\right)}^{2k-2\ell+1}}\\[.4em]
    &\qquad=\frac{\II_{\{k=\ell\}}}{t+\frac12}-\II_{\{k>\ell\}}\frac{{\left(t-\frac12\right)}^{2k-2\ell-1}}{{\left(t+\frac12\right)}^{2k-2\ell+1}}.
\end{align*}
That this also holds for~$t=\frac12$ is anecdotal
and not difficult to check (e.g.,
using $L_{2\ell}'-L_{2\ell+1}'=L_{2\ell}$
and the orthogonality relations~\eqref{eq:laguerre_orthogonality}).
Details are left to the interested reader.
\end{proof}

Before we move to the proof of Theorem~\ref{th:laguerre}, we need the following two calculations.

\begin{lemma}\label{lem:lue_integrals}
    We have
    \[\int_0^\infty\frac{\dd t}t\left(\frac2{1+t}-\frac1{t+\frac12}\right)
        =2\log2,\]
    and for any integer~$m\ge0$,
    \[
        \int_0^\infty\frac{\dd t}t\left(\frac4{1+t}+\frac{{\left(t-\frac12\right)}^{2m-1}}{{\left(t+\frac12\right)}^{2m+1}}\right)
        =4\Bigl(\log2+2\oH_{m}\Bigr),
    \]
    where $\oH_m\ceq1+\frac13+\cdots+\frac1{2m-1}$.
\end{lemma}
\begin{proof}
    The first part of the lemma is clear since
    \begin{equation}
    \int_0^\infty\frac{\dd t}t\left(\frac2{1+t}-\frac1{t+\frac12}\right)
    =2{\left[\log{\frac{1+2t}{1+t}}\right]}_{t=0}^\infty
    =2\log 2.\label{eq:log_integral}
    \end{equation}
    Next, the change of variable $u\gets1/(2+t)$ gives
    \begin{align}
    \int_0^\infty\frac{\dd t}t\left(\frac2{t+\frac12}+\frac{{\left(t-\frac12\right)}^{2m-1}}{{\left(t+\frac12\right)}^{2m+1}}\right)
    &=4\int_0^1\frac{1-u{(2u-1)}^{2m-1}}{1-u}\,\dd u\notag\\[.4em]
    &=4\int_0^1\frac{1-{(2u-1)}^{2m-1}}{1-u}\,\dd u\notag\\[.4em]
    &=8\sum_{k=0}^{2m-2}\int_0^1{(2u-1)}^k\,\dd u\notag\\[.4em]
    &=8\oH_m,\label{eq:harmonic}
    \end{align}
    using that $2\int_0^1{(2u-1)}^k\,\dd u=\bigl[1+{(-1)}^k\bigr]/(k+1)$.
    The second part of the lemma then follows by doubling~\eqref{eq:log_integral}
    and adding~\eqref{eq:harmonic}.
\end{proof}

\begin{lemma}\label{lem:harmonic_identity}
    We have
\[16\sum_{0\le k<\ell<n}
{(-1)}^{k+\ell}\binom{n-1}k\binom{n-1}\ell\binom{k-\frac12}n\binom{\ell-\frac12}n\oH_{\ell-k}
    =H_n-\frac1{2n}-\frac12,
\]
where we recall that $H_n\ceq1+\frac12+\cdots+\frac1n$ and $\oH_m\ceq1+\frac13+\cdots+\frac1{2m-1}$.
\end{lemma}
\begin{proof}
    Let us call~$A_n$ the left-hand side.
    Observe from~\eqref{eq:laguerre_closedform} and Lemma~\ref{lem:LUE density} that
$L_k(0)=1$ and $\mL(0)=1$, that is
    \begin{equation}
        \sum_{k=0}^{n-1}p_k=1,\quad\text{where}\enspace p_k\ceq2{(-1)}^{n-k}\binom{n-1}k\binom{k-\frac12}n\ge0.\label{eq:pk}
    \end{equation}
    We also observe for any integer~$m\in\ZZ$ the telescoping summation
    \begin{equation}\sum_{j=0}^\infty\left(\frac1{j+\frac12}-\frac1{j+\frac12+m}\right)=
    \sum_{j=0}^{\lvert m\rvert-1}\frac1{j+\frac12}=2\oH_{\lvert m\rvert},\label{eq:telescoping}
    \end{equation}
    so that we can get rid of the restriction~$\ell>k$ and write
    \begin{align*}
        A_n
        &=\sum_{j=0}^\infty\sum_{0\le k,\ell<n}p_kp_\ell\left(\frac1{j+\frac12}-\frac1{j+\frac12+\ell-k}\right)
        \\[.4em]
        &=\sum_{j=0}^\infty\left(\frac1{j+\frac12}-\sum_{k=0}^{n-1}p_k\EE\frac1{j+\frac12+X-k}\right),
    \end{align*}
    where~$X$ is a random variable with law~${(p_\ell)}_{0\le\ell<n}$. We now compute the
    Cauchy-Stieltjes transform
    \[S(z)\ceq\EE{\frac1{z+X}}=\sum_{\ell=0}^{n-1}\frac{p_\ell}{z+\ell},\]
    which has simple poles at~$z=-\ell$ for $0\le\ell<n$. Written as a rational fraction,
    $S(z)=P(z)/Q(z)$ where $Q(z)\ceq z(z+1)\cdots(z+n-1)$ and,
    by uniqueness of the partial fraction decomposition, $P(z)$ is a
    polynomial of degree less than~$n$ determined by the relations
    \[p_\ell=\lim_{z\to-\ell}(z+\ell)S(z)=\frac{P(-\ell)}{\prod\limits_{\substack{0\le k<n\\k\neq\ell}}(k-\ell)}=\frac{{(-1)}^\ell P(-\ell)}{\ell!(n-1-\ell)!},\]
    for all $0\le \ell<n$, that is, after~\eqref{eq:pk},
    \begin{align*}
    P(-\ell)&=2{(-1)}^n(n-1)!\binom{\ell-\frac12}n\\[.4em]
    &=\frac2n\left(-\ell+\tfrac12\right)\cdots\left(-\ell+n-\tfrac12\right).
    \end{align*}
    This entails the equality of polynomials $P(z)=\frac2n\bigl((z+\frac12)\cdots(z+n-\frac12)-Q(z)\bigr)$,
    and we have thus found
    \begin{equation}
        S(z)=\EE{\frac1{z+X}}=\frac2n\Bigl(R(z)-1\Bigr),\quad\text{where}\enspace
    R(z)\ceq\frac{z+\frac12}z\times\cdots\times\frac{z+n-\frac12}{z+n-1},\label{eq:stieltjes}
    \end{equation}
    valid for $-z\notin\{0,\ldots,n-1\}$.
    As a result, 
    noting that $S(j+\frac12-k)=-\frac2n$ for $j<k<n$,
    \begin{align*}
        A_n
        &=\EE{\sum_{j=0}^\infty\left(\frac1{j+\frac12}-S(j+\tfrac12-X)\right)}\\[.4em]
        &=\EE{\left[\sum_{j=0}^{X-1}\left(\frac1{j+\frac12}+\frac2n\right)+\sum_{j=0}^\infty\left(\frac1{j+\frac12+X}-S(j+\tfrac12)\right)\right]}\\[.4em]
        &=\frac2n\EE X+2\EE{\oH_X},
    \end{align*}
    by~\eqref{eq:telescoping}. To compute~$\EE X$, we let~$z\to\infty$ in
    \[\EE{\frac{zX}{z+X}}=z\left(1-\EE{\frac z{z+X}}\right)=z\Bigl(1-zS(z)\Bigr),\]
    where, from~\eqref{eq:stieltjes},
    \[S(z)=\frac1z-\frac{n-1}{4z^2}+o\left(\frac1{z^2}\right).\]
    Thus $\EE X=\frac{n-1}4$. It remains to evaluate $2\EE{\oH_X}$.
    In this direction, we rewrite
    \[p_k\ceq2{(-1)}^{n-k}\binom{n-1}k\binom{k-\frac12}n=\frac{{(-1)}^{n-1}}n\binom{-\frac12}k\binom{-\frac32}{n-1-k},\]
    and differentiate Vandermonde's convolution identity~\cite[(3.1)]{Gould72}
    \[\sum_{k=0}^{n-1}\binom xk\binom y{n-1-k}=\binom{x+y}{n-1},\quad x,y\in\RR,\]
    with respect to~$x$ to get (recall~\eqref{eq:binomial})
    \[\sum_{k=0}^{n-1}\binom xk\binom y{n-1-k}\sum_{i=0}^{k-1}\frac1{x-i}=\binom{x+y}{n-1}\sum_{i=0}^{n-2}\frac1{x+y-i}.\]
    Plugging in $x\ceq-\frac12$ and $y\ceq-\frac32$ produces
    \[\sum_{k=0}^{n-1}\binom{-\frac12}k\binom{-\frac32}{n-1-k}\sum_{i=0}^{k-1}\frac1{i+\frac12}
    =\binom{-2}{n-1}\sum_{i=0}^{n-2}\frac1{i+2}
    ={(-1)}^{n-1}\,n\Bigl(H_n-1\Bigr),\]
    hence
    \[2\EE{\oH_X}=H_n-1.\]
    We conclude that
    \[A_n=\frac2n\EE X+2\EE{\oH_X}=\frac{n-1}{2n}+H_n-1=H_n-\frac12-\frac1{2n},\]
    as stated.
\end{proof}

We are now in place to prove Theorem~\ref{th:laguerre}. 
\begin{proof}[Proof of Theorem~\ref{th:laguerre}]
We start by combining Lemmas~\ref{lem: step1-lue} and~\ref{lem: step2-lue} to write 
\[
\En(\lL)
=\log(2n)+\gamma
    -\int_0^\infty\frac{\dd t}t
    \left(\frac1{1+t}-J(t)
    \right),
\]
where 
\[
J(t)=2\sum_{0\le k,\ell<n}
    {(-1)}^{k+\ell}\binom{n-1}k\binom{n-1}\ell\binom{k-\frac12}n\binom{\ell-\frac12}n \left\{
    \frac{\II_{\{k=\ell\}}}{t+\frac12}-\II_{\{k>\ell\}}\frac{{\left(t-\frac12\right)}^{2k-2\ell-1}}{{\left(t+\frac12\right)}^{2k-2\ell+1}}\right\}.
\]
Recall from~\eqref{eq:pk} that
\begin{align}
2\sum_{k=0}^{n-1}
{(-1)}^{n-k}\binom{2k}k\binom{k-\frac12}n&=1,&&&&\label{eq:lue_sum}
\intertext{whence, by squaring,}
4\sum_{0\le k,\ell<n}
{(-1)}^{k+\ell}\binom{n-1}k\binom{n-1}\ell\binom{k-\frac12}n\binom{\ell-\frac12}n
&=1.\label{eq:lue_doublesum}
\end{align}
Exploiting the symmetry in~$k$ and~$\ell$, we write
\begin{equation}
\bullet=2\sum_{0\le k,\ell<n}
{(-1)}^{k+\ell}\binom{n-1}k\binom{n-1}\ell\binom{k-\frac12}n\binom{\ell-\frac12}n
\Bigl[\II_{\{k=\ell\}}{2\bullet}+\II_{\{k>\ell\}}4\bullet\Bigr],
\label{eq:lue_doublesum2}
\end{equation}
to get
\begin{align*}
\int\frac{\dd t}t
    \left(\frac1{1+t}-J(t)
    \right)&=2\sum_{0\le k,\ell<n}
{(-1)}^{k+\ell}\binom{n-1}k\binom{n-1}\ell\binom{k-\frac12}n\binom{\ell-\frac12}n\\[.4em]
&\quad\times\int_0^\infty\frac{\dd t}t\left\{\II_{\{k=\ell\}}\left(\frac2{1+t}-\frac 1{t+\frac12}\right)
+\II_{\{k>\ell\}}\left(\frac4{1+t}+\frac{{\left(t-\frac12\right)}^{2k-2\ell-1}}{{\left(t+\frac12\right)}^{2k-2\ell+1}}\right)\right\}.
\end{align*}
We gather from Lemma~\ref{lem:lue_integrals} the values
\begin{align*}
\int_0^\infty\frac{\dd t}t\left(\frac2{1+t}-\frac1{t+\frac12}\right)
&=2\log 2,
\intertext{and, for all integers~$k>\ell$,}
\int_0^\infty\frac{\dd t}t\left(\frac4{1+t}+\frac{{\left(t-\frac12\right)}^{2k-2\ell-1}}{{\left(t+\frac12\right)}^{2k-2\ell+1}}\right)
&=4\Bigl(\log2+2\oH_{k-\ell}\Bigr),
\end{align*}
where~$\oH_m\ceq\sum_{i=1}^m\frac1{2i-1}$. Factorizing out~$\log2$ and recalling~\eqref{eq:lue_doublesum}
and~\eqref{eq:lue_doublesum2}, we then arrive at
\begin{align*}
\En(\lL)
&=\gamma+\log n-16\sum_{0\le k<\ell<n}
{(-1)}^{k+\ell}\binom{n-1}k\binom{n-1}\ell\binom{k-\frac12}n\binom{\ell-\frac12}n\oH_{\ell-k}.
\end{align*}
We conclude using  Lemma~\ref{lem:harmonic_identity}
that
\[\En(\lL)=\frac12-\left(H_n-\log n-\gamma-\frac1{2n}\right)=2\En(\lH),\]
giving $\En^{(1)}(\lL)=1+2\En(\lH)=2\En^{(2)}_2(\lH)$.
\end{proof}

\section{Numerical experiments}\label{se:numexp}
We include Monte Carlo simulations supporting positive answers to
Questions~\ref{qu: semi-circle},~\ref{qu: circle}, and~\ref{qu: MP}.
We consider several centered distributions for the matrix
coeffcients $\xi_{ij}$, according to the table below:
$\textsf{Rademacher}$, symmetric generalized gamma
$\textsf{SGG}(d,p)$ with $d,p>0$, and heavy-tailed $\textsf{Heavy}(\alpha)$
with~$\alpha>2$. The coefficients are always normalized so that $\var(\xi_{ij})=1$. In the discrete case (\textsf{Rademacher}), simulations resulting in an infinite value of~$\En$ have been rejected. 

\medskip

\begin{center}
\begin{tabular}{c|c|c}
  \textbf{Distribution} & \textbf{Support} & \textbf{Density}\\[.5em]
  \hline\hline&\\
  $\textsf{Rademacher}$ & $\{-1,1\}$ & $\frac12$\\[1em]
  \hline&\\
  $\textsf{SGG}(d,p)$ & $\RR$ & $\frac p{2\Gamma(\frac dp)}\,{\lvert x\rvert}^{d-1}\ee^{-{\lvert x\rvert}^p}$\\[1em]
  \hline&\\
  $\textsf{Heavy}(\alpha)$ & $\RR$ & $\frac\alpha 2\,{(1+\lvert x\rvert)}^{-\alpha-1}$\\[1em]
\end{tabular}
\end{center}

\medskip

\subsection{Positive evidence to Question~\ref{qu: semi-circle}}
Figure~\ref{fig:wigner} 
represents, for different distributions of the matrix coefficients~$\xi_{ij}$ above,
Monte Carlo simulations of the logarithmic energy
$\En^{(2)}_1(\dE\mu_{\frac1{\sqrt n}M_n})$
associated with the random
Wigner matrix $M_n\ceq(\xi_{ij})_{1\le i,j\le n}$.

\begin{figure}[p]
  \includegraphics[width=\textwidth]{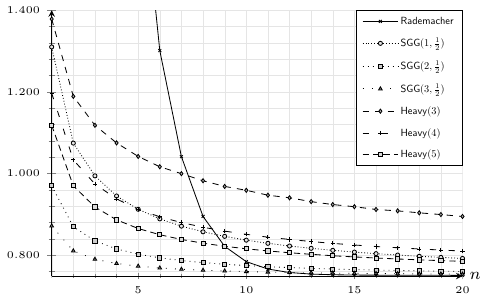}
  \caption{Monte Carlo simulations of the logarithmic energy $\En^{(2)}_1(\dE\mu_{\frac1{\sqrt n}M_n})$
  for several examples of random Wigner matrices~$M_n$.}\label{fig:wigner}
\end{figure}

We conjecture that Question~\ref{qu: semi-circle}
has a positive answer beyond the scope of Wigner matrices.
Figure~\ref{fig:hBeta} displays,
for different values of~$\beta>0$,
Monte Carlo simulations of the logarithmic energy
$\En^{(2)}_1(\dE\mu_{H_n^{(\beta)}})$ where
$H_n^{(\beta)}$ is the (normalized) tridiagonal matrix model
of the $\beta$-Hermite Gaussian ensemble~\cite[\S~5.2]{Dumitriu03}.
Specifically,
\[H_n^{(\beta)}\sim\frac1{\sqrt{2+\beta(n-1)}}
\begin{pmatrix}
  a_1&b_1\\
  b_1&a_2&b_2\\
  &\ddots&\ddots&\ddots\\
  &&b_{n-2}&a_{n-1}&b_{n-1}\\
  &&&b_{n-1}&a_n
\end{pmatrix},
\]
where $a_1,\ldots,a_n,b_1,\ldots,b_{n-1}$ are mutually independent
random variables,
$a_i$ has the normal~$\cN(0,2)$ distribution and~%
$b_i$ has the chi distribution with parameter $(n-i)\beta$.
We recall that~$H_n^{(2)}$ shares the same eigenvalue distribution as
$A_n^{\GUE}$, see~\cite[Theorem~5.2.1]{Dumitriu03}.

\begin{figure}[p]
  \includegraphics[width=\textwidth]{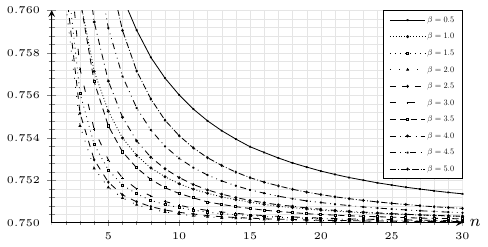}
  \caption{Monte Carlo simulations of the logarithmic energy $\En^{(2)}_1(\dE\mu_{H^{(\beta)}_n})$
  associated with the $\beta$-Hermite Gaussian ensemble $H^{(\beta)}_n$
  for different values of~$\beta$.\label{fig:hBeta}}
\end{figure}

\subsection{Positive evidence to Question~\ref{qu: circle}}
Figure~\ref{fig:girko}
represents,
for different distributions of the matrix coefficients~$\xi_{ij}$,
Monte Carlo simulations of the logarithmic energy
$\En^{(2)}_{\frac12}(\dE\mu_{\frac1{\sqrt n}M_n})$,
associated with the i.i.d.\ matrix
$M_n\ceq(\xi_{ij})_{1\le i,j\le n}$.

\begin{figure}[p]
  \includegraphics[width=\textwidth]{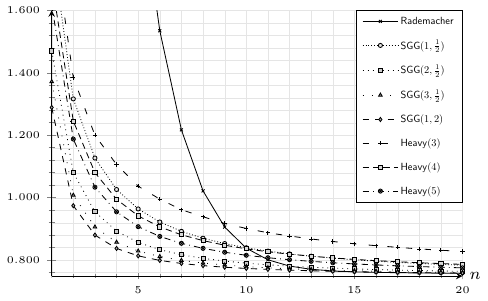}
  \caption{Monte Carlo simulations of the logarithmic energy $\En^{(2)}_{\frac12}(\dE\mu_{\frac1{\sqrt n}M_n})$
  for several examples of random i.i.d.\ matrices~$M_n$. Note that $\textsf{SGG}(1,2)$ corresponds to the real Ginibre ensemble.}\label{fig:girko}
\end{figure}

\subsection{Positive evidence to Question~\ref{qu: MP}}
Figure~\ref{fig:mp}
represents,
for different distributions of the matrix coefficients~$\xi_{ij}$,
Monte Carlo simulations of the logarithmic energy
$\En^{(1)}(\dE\mu_{\frac1nX_nX_n^*})$
where
$X_n\ceq(\xi_{ij})_{1\le i,j\le n}$.

\begin{figure}[p]
  \includegraphics[width=\textwidth]{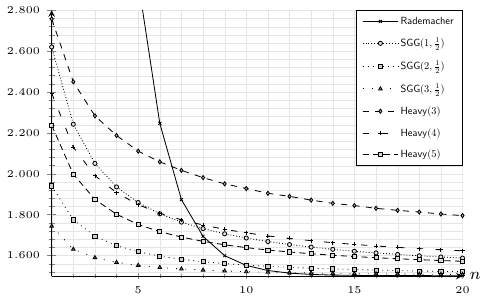}
  \caption{Monte Carlo simulations of the logarithmic energy $\En^{(1)}(\dE\mu_{\frac1nX_nX_n^*})$
  for several examples of random i.i.d.\ matrices~$X_n$.\label{fig:mp}}
\end{figure}

\bibliographystyle{abbrv}
\bibliography{monotonicity}

\end{document}